\documentclass[a4paper]{article}

\usepackage[utf8]{inputenc}
\usepackage{amsmath, amssymb, amsthm, xypic, textcomp, csquotes}

\topmargin = - 1.5 cm
\newcommand{\Co}{\mathbb{C}}
\newcommand{\GLn}{\mathrm{GL}_n(\mathbb{C})}
\newcommand{\im}{\mathrm{im} \,}

\newtheorem{theorem}{Theorem}[section]
\newtheorem{predl}[theorem]{Proposition}
\newtheorem{definition}[theorem]{Definition}
\newtheorem{lemma}[theorem]{Lemma}
\newtheorem{statement}[theorem]{Lemma}
\newtheorem{corollary}[theorem]{Corollary}

\begin{document}
\begin{center}
{\LARGE \textbf{Automorphisms of Hopf manifolds}} \\[5pt]

{\Large Anna Savelyeva}
\end{center}

\begin{abstract}
We prove that the automorphism groups of Hopf manifolds are Jordan.
\end{abstract}

\section{Introduction}
We are interested in the automorphism groups of complex manifolds, but they are often very complicated and there is no general approach to them. However we can try to consider only their finite subgroups. The finite groups are easier to study if they contain abelian subgroups of small indices, as the structure of finite abelian groups is well known. That is where the next definition comes from.

\begin{definition}[{see \cite[Definition 2.1]{Jordan}}] \label{Jordan}
The group $G$ is called Jordan if there exists a constant $C$, such that for any finite subgroup $H \subset G$ there exists an abelian subgroup of $H$ whose index in $H$ is less than $C$.
\end{definition}

Note that Definition $\ref{Jordan}$ is not exactly the same as in \cite{Jordan}. Unlike \cite{Jordan} we do not require the abelian subgroup to be normal. However these definitions are equivalent, because if an abelian subgroup of index $k$ exists, then it is easy to construct a normal abelian subgroup of index at most $k^k$. This bound can be easily improved to $k!$ (see \cite[Corollary 1.2]{k^2}), and, with more effort, to $k^2$ (see \cite[Theorem 1.41]{k^2}). \\[-5pt]

Camille Jordan has proved the following result:

\begin{theorem}[{see \cite[Theorem 36.13]{GLn}}] \label{GLn}
The group $\GLn$ is Jordan.
\end{theorem}

It is natural to ask whether there exists a compact complex manifold whose automorphism group is not Jordan. There is no such manifold known. Now the Jordan property of an automorphism group is proved for compact K\"ahler manifolds (see \cite{Kim}), all compact complex manifolds of dimension $2$ (see \cite{surfaces}) and all Moishezon manifolds of dimension $3$ (see \cite{moish}). There is also a result that generalizes the two results mentioned above (see \cite{Meng}). Note that there exist compact real manifolds with non-Jordan automorphism groups (see \cite{contr}).

Here we prove the Jordan property of the automorphism group of one more type of compact complex manifolds.

\begin{definition}[{see \cite[Introduction]{hoph}}] \label{hoph}
A compact complex manifold of dimention $n \geqslant 2$ is called a Hopf manifold if its universal cover is isomorphic to $\Co^n \backslash 0$.
\end{definition}

\begin{theorem} \label{main}
Let $Y$ be a Hopf manifold. Then the group $\mathrm{Aut}(Y)$ is Jordan.
\end{theorem}

In the case of Hopf surfaces ($n = 2$) Theorem \ref{main} is proved in \cite[Lemma 6.4]{surfaces} and also follows from \cite[Theorem 8.1]{surfaces2}. \\[-5pt]

In Sections $2$ and $3$ we prove several auxiliary statements from group theory and linear algebra, in Section $4$ we recall the structure of Hopf manifolds, and in Section $5$ we prove Theorem \ref{main}. \\

I would like to thank my advisor Constantin Shramov for stating the problem, useful discussions and constant attention to this work. I also want to thank Alexander Gaifullin for his help with incorporating topological statements.

\section{Linear algebra}
In this section we prove several elementary statements about commuting matrices.

Let us remind the reader the definition of a root subspace:

\begin{definition}
Suppose $V$ is a vector space, $K$ is a linear operator on $V$, and $\lambda$ is an eigenvalue of $K$. Then the space $$\underset{j}{\bigcup} \ker (K - \lambda)^j$$ is called the root subspace of $K$ corresponding to the eigenvalue $\lambda$.
\end{definition}

\begin{statement} \label{st}
Let $K \in \GLn$ be a matrix. Suppose $A \in \GLn$ commutes with $K$. Then the root subspaces of $K$ are $A$-invariant.
\end{statement}

\begin{proof}
Let $\lambda_1, \ldots, \lambda_k$ be the eigenvalues of $K$. Denote by $K_i$ the root subspace of $K$ that corresponds to an eigenvalue $\lambda_i$. By the definition $$K_i = \underset{j}{\bigcup} \ker (K - \lambda_i)^j.$$ Let us prove that $\ker(K-\lambda_i)^j$ is $A$-invariant for all $j$. Let $v$ be a vector in $\ker(K-\lambda_i)^j$. Then $(K-\lambda_i)^j(v) = 0$ and therefore $$A(K-\lambda_i)^j(v) = (K-\lambda_i)^jA(v) = 0.$$ Consequently, $A(v)$ belongs to $\ker(K-\lambda_i)^j$, so $\ker(K-\lambda_i)^j$ is $A$-invariant, and so the root subspaces of $K$ are also $A$-invariant.
\end{proof}

\begin{lemma} \label{lemm}
Let $K$ belong to $\GLn$ and $m$ be a natural number. Then there exists $\widehat{K} \in \GLn$ such that $\widehat{K}^m = K$, and each $A \in \GLn$ which commutes with $K$ also commutes with $\widehat{K}$.
\end{lemma}

\begin{proof}
First we will prove the lemma for the matrices with all eigenvalues equal to each other.

Suppose $K$ is such a matrix. Then $K = \lambda E + N$, where $\lambda$ is the eigenvalue of $K$ and $N$ is a nilpotent matrix.

Consider the Taylor series for $\sqrt[m]{\lambda + x}$. If we replace $x$ by $N$, then all terms of the series except the finite number of them will be equal to $0$, as $N$ is a nilpotent. So $\sqrt[m]{K}$ exists and is equal to $a_0 E + a_1 N + \ldots + a_n N^n$, where $a_1, \ldots a_n$ are complex numbers. Let $A$ be a matrix commuting with $K = \lambda E + N$. Then $A$ commutes with $N$, as all scalar matrices lie in the center of $\GLn$. Consequently it commutes with any polynomial in $N$, in particular with $\sqrt[m]{K}$.

Now let $K$ be any matrix in Jordan normal form. Let's split $K$ into the root subspaces $K_i$.

Let
$$\widehat{K} = 
      \begin{pmatrix}
      \widehat{K}_1 &0 &\ldots &0 \\
      0 &\widehat{K}_2 &\ldots &0 \\
      \vdots &\vdots &\ddots &\vdots \\
      0 &0 &\ldots &\widehat{K}_k
      \end{pmatrix},$$
where $\widehat K_i$ is a matrix, constructed from $K_i$ and $m$ by the first part of the proof.

Then $\widehat{K}^m = K$ and $\widehat{K}A = A \widehat{K}$ for any $A$, that commutes with $K$, because if $A$ and $K$ commute, then the root subspaces of $K$ are $A$-invariant, and so the restrictions of $A$ and $K$ on these subspaces commute.
\end{proof}

\section{Group theory}

In this section we prove several statements about groups. We introduce several notations: denote by $e$ the neutral element of a group, $Z(G)$ is the center of a group $G$. Denote by $[A \colon B]$ the index of  subgroup $B$ of a group $A$.

The next lemma is sometimes called Schur's Lemma.

\begin{lemma} [{see \cite[15.1.13]{kommutant}}] \label{shur}
Suppose $M$ is a group, and $\Gamma \subset M$ is a subgroup of $Z(M)$ isomorphic to $\mathbb{Z}$. Suppose $H = M/\Gamma$ is finite. Then the commutator subgroup $\mathcal{K}$ of $M$ is finite.
\end{lemma}

We will use Lemma \ref{shur} to prove the following proposition.

\begin{predl} \label{group_add1}
In the conditions of Lemma \ref{shur} there exists a normal subgroup $R \subset M$, such that $M/R \cong \mathbb{Z}$.
\end{predl}

\begin{proof} [First proof {\upshape (see \cite[Lemma 4.1]{pr})}]

Denote by $\mathcal{K}$ the commutator subgroup of $M$. By Lemma \ref{shur} it is finite. Consider $A = M/\mathcal{K}$. Note that $A$ is an infinite finitely-generated abelian group, so it is isomorphic to ${\mathbb{Z}^k \times B}$, where $B$ is a finite group. The quotient $M / \Gamma$ is finite and $\Gamma$ is isomorphic to $\mathbb{Z}$, so $k = 1$. Then $M/\mathcal{K} \cong \mathbb{Z} \times B$. Let $R$ be the preimage of $B$ by taking a quotient by $\mathcal{K}$. Then $R$ is a finite normal subgroup of $M$ such that $M/R$ is isomorphic to $\mathbb{Z}$.
\end{proof}

\begin{proof}[Second proof]
Let $\mu \colon M \to H = M / \Gamma$ be the canonical projection. Denote by $n$ the order of group $H$ and by $\rho$ a map from $M$ to $\Gamma$, such that $\rho (a) = a^n$. Note that the image of this map is a subset of $\Gamma$, as $\mu(\rho(a)) = \mu(a^n) = \mu(a)^n = e$ for all $a$. Let us prove that $\rho$ is a group homomorphism. Obviously, $\rho(e) = e$. The only thing left to prove is that $\rho(ab) = \rho(a)\rho(b)$ for all $a$, $b$.

Let's fix $a$, $b \in M$. For each element of $H$ let's choose any element from its preimage by $\mu$. Denote those elements by $x_1, x_2, \ldots, x_n$. The set of values of $x_1a,x_2a, \ldots,  x_na$ is a set of preimages of all elements of $H$, and so they are equal to $x_1a_1, \ldots, x_na_n$ (possibly in a different order), where $a_i$ are elements of $\Gamma$. Let $\sigma_a$ be a permutation of numbers from $1$ to $n$, such that $x_ia = x_{\sigma_a(i)}a_{\sigma_a(i)}$ for all $i$. Similarly define $b_i$ and $\sigma_b$.

Let's decompose $\sigma_a$ into disjoint cycles. For a cycle $(i \;\, \sigma_a(i) \; \ldots \; \sigma_a^{t-1}(i))$ we get
$$a^t = a_{\sigma_a(i)}a_{\sigma_a^2(i)} \ldots a_{\sigma_a^{t - 1}(i)}a_i x_i^{-1}x_{\sigma_a(i)}x_{\sigma_a(i)}^{-1}x_{\sigma_a^2(i)}x_{\sigma_a^2(i)}^{-1} \ldots x_{\sigma_a^{t-1}(i)} x_{\sigma_a^{t-1}(i)}^{-1} x_i = a_ia_{\sigma_a(i)}\ldots a_{\sigma_a^{t-1}(i)}.$$

Multiplying these expressions for all disjoint cycles (including those of length 1) we get $a^n = a_1 a_2 \ldots a_n$. Similarly $b^n = b_1b_2 \ldots b_n$.

Similarly with $a$ and $b$ we find the value of $(ab)^n$. Let $\sigma = \sigma_b \circ \sigma_a$. Then $$x_iab = a_{\sigma_a(i)}x_{\sigma_a(i)}b = a_{\sigma_a(i)}b_{\sigma(i)}x_{\sigma(i)}.$$ Decompose $\sigma$ into disjoint cycles. For a cycle $(i \;\, \sigma_a(i) \; \ldots \; \sigma_a^{t-1}(i))$ we get $$(ab)^t = a_{\sigma_a(i)}a_{\sigma_a \circ \sigma(i)} \ldots a_{\sigma_a \circ \sigma^{t-1}(i)}b_{\sigma(i)}b_{\sigma^2(i)}\ldots b_{\sigma^{t-1}(i)}b_ix_i^{-1}x_{\sigma(i)}x_{\sigma(i)}^{-1}x_{\sigma^2(i)}\ldots x_{\sigma^{t-1}(i)}^{-1}x_{\sigma^{t-1}(i)} x_i.$$

Multiplying for all disjoint cycles we achieve $(ab)^n = a_1a_2\ldots a_nb_1 b_2 \ldots b_n = a^nb^n$. So we proved that $\rho$ is a homomorphism.

Note that each element of $\Gamma$, apart from $e$, has infinite order, and so it doesn't map to $e$. So the kernel of $\rho$ intersects with $\Gamma$ only by $e$, and so it is finite, as $M/\Gamma$ is finite. Consequently $\ker \rho$ is a finite normal subgroup, such that $M/ \ker \rho \cong \mathbb{Z}$.
\end{proof}

Now we prove one more additional lemma.

\begin{statement} \label{late}
Suppose $H$ is a finite group, $\xi \colon H \to H'$ is the canonical projection onto a quotient of $H$ by a subgroup $N$ of $Z(H)$, and $\psi \colon H \to \Gamma'$ is a homomorphism from $H$ to an abelian group $\Gamma'$. Suppose the kernels of $\xi$ and $\psi$ intersect only by the neutral element. Denote by $\mathcal{A}$, $\mathcal{A}'$ abelian subgroups of the smallest indices of $H$ and $H'$ respectively. Then $[H \colon \mathcal{A}] = [H' \colon \mathcal{A}']$.
\end{statement}

\begin{proof}
Note that $\mathcal{A} \supset N$, because otherwise the group generated by $N$ and $\mathcal{A}$ would be abelian (as $N$ is a subgroup of $Z(H)$) and would be strictly larger than $\mathcal{A}$, which is not possible by the definition of $\mathcal{A}$. Moreover, $\xi(\mathcal{A})$ is an abelian group, and so $[H \colon \mathcal{A}] \geqslant [H' \colon \mathcal{A}']$. We will prove that the preimage of  $\mathcal{A}'$ by $\xi$ is abelian.

Consider any two elements $a, b \in \xi^{-1}(\mathcal{A}')$. Let's prove that they commute. Consider $c = aba^{-1}b^{-1}$. Note that $$\psi(c) = \psi(a)\psi(b)\psi(a)^{-1}\psi(b)^{-1} = e,$$ because $\Gamma'$ is abelian. So $c \in \ker(\psi)$. Similarly $$\xi(c) = \xi(a)\xi(b)\xi(a)^{-1}\xi(b)^{-1} = e,$$ because $\xi(a)$, $\xi(b)$ are the elements of $\mathcal{A}'$. So $c \in \ker(\xi)$. However $\ker(\psi) \cap \ker(\xi) = e$, and so $c = e$. That proves that $a$ and $b$ commute. Consequently $\xi^{-1}(\mathcal{A}')$ is an abelian group. So $[H \colon \mathcal{A}] \leqslant [H' \colon \mathcal{A}']$, and so $[H \colon \mathcal{A}] = [H' \colon \mathcal{A}']$.
\end{proof}

The next proposition is the crucial one to prove the main result of this section.

\begin{predl} \label{group_add2}
Let $M$ be a subgroup of $\GLn$, and $\Gamma \subset Z(M)$ be a subgroup of $M$, isomorphic to $\mathbb{Z}$. Suppose $H = M/\Gamma$ is finite. Let $\mathcal{A}$ be an abelian subgroup of $H$ of the smallest index. Then there exists a finite group 
$H' \subset \GLn$, such that $[H \colon \mathcal{A}] = [H' \colon \mathcal{A}']$, where $\mathcal{A}'$ is an abelian subgroup of $H'$ of the smallest index.
\end{predl}

\begin{proof}
Let $\mu \colon M \to H = M/\Gamma$ be the canonical projection. By Proposition \ref{group_add1} there exists a normal subgroup $R \subset M$, such that $M/R \cong \mathbb{Z}$. Denote by $\rho \colon M \to M/R$ the canonical projection (compare with the second proof of Proposition \ref{group_add1}). Let $K$ be a generating element of $\Gamma$ and $\rho(K) = m$.

Let $\widehat{K}$ be the matrix we get from $K$ and $m$ by Lemma \ref{lemm}. Consider a homomorphism $\varphi \colon M \to \GLn$ such that $A \mapsto A \widehat{K}^{\rho(A)}$, where $A$ is an element of $M$.

Let's denote the image of $\varphi$ by $H'$ and prove, that this is the group we are looking for in Proposition \ref{group_add2}. First we will prove that $H' \cong H/N$, where $N$ is a subgroup of $Z(H)$. Suppose $B$ is an element of $\ker \varphi$. Then there exists an integer $l$ such that $B = \widehat{K} ^ l$. So $B \in Z(M)$.

Let's prove that $K \in \ker \varphi$. As $\rho(K) = m$ we conclude that $$\varphi(K) = \widehat{K}^{-m}K = K^{-1}K = e.$$ So $\varphi = \xi \circ \mu$, where $\xi \colon H \to \im \varphi$ is such a map that $\ker \xi = \mu(\ker \varphi)$.

Now we know that $H' = H/\mu(\ker\varphi)$. The only thing left to prove is that $\mu(\ker\varphi)$ is a subgroup of $Z(H)$. It is so because any surjection maps the center of the preimage to a subgroup of the center of the image.

Let's prove that groups $H$ and $H'$ satisfy the conditions of Lemma \ref{late}. Note that $\xi \colon H \to H'$ is the canonical projection to a quotient of $H$ by a subgroup of $Z(H)$. Now we will construct a homomorphism $\psi$ from $H$ to an abelian group $\Gamma'$.

Consider $R \subset M$. Let $R'$ be an image of $R$ by $\mu$. Note that $R'$ is a normal subgroup of $H$, as an image of a normal subgroup by surjection is normal. Set $\Gamma' = H/R'$, and $\psi \colon H \to \Gamma'$ the canonical projection. Let $\psi'$ be the canonical projection from $\im \rho$ to $\im \rho / \rho(\Gamma)$. Note that $\psi' \circ \rho = \psi \circ \mu$, as the kernels of these maps coincide. Also, $\im(\psi' \circ \rho)$ is a cyclic abelian group $\mathbb{Z}/m\mathbb{Z}$, as $\im \rho \cong \mathbb{Z}$, and $\rho(K) = m$. So $\Gamma'$ is an abelian group. So we have proved that the diagram below is commutative. \\[-25pt]

$$\xymatrix@1{
\mathbb{Z} \ar[d]_(0.45){\psi'} & M \ar[d]_(0.45)\mu \ar[l]_\rho \ar[rd]^(0.45)\varphi \\
\Gamma' & H \ar[l]_(0.45)\psi \ar[r]^(0.4){\xi} & H'
}$$

It is left to prove that $\ker \xi \cap \ker \psi = e$. Note that $\ker \xi = N$, and $\ker \psi = R'$. Let $x$ be the element in an intersection of $N$ and $R'$. Let $x'$ be any preimage of $x$ by $\mu$. We know that $\mu(x') \in N$, and so $x' \in \ker \varphi$. Consequently $x' = \widehat{K} ^ {\rho(x')}$. But $\mu(x') \in R'$, so $(\psi' \circ \rho)(x') = e$, and so $\rho(x')$ is divisible by $m$. So $x' = \widehat{K} ^ {ml} = K^l$, where $l$ is an integer. So $x'$ is an element of the kernel of $\mu$, and so $x = e$. We proved that $H$ and $H'$ satisfy the condition of Lemma \ref{late}. So $H'$ is the group we need.
\end{proof}

Now we can prove a theorem we will use to prove Theorem \ref{main}.

\begin{theorem} \label{group_main}
For every natural $n$ there exists a constant $C$ satisfying the following condition: let $M$ be a subgroup of $\GLn$, which contains a normal subgroup $\Gamma$ of finite index, such that $\Gamma \cong \mathbb{Z}$. Denote by $c_M$ the smallest of the indices of abelian subgroups of $M/\Gamma$. Then $c_M < C$.
\end{theorem}

Note that in this theorem we don't need $\Gamma$ to be a subgroup of the center of $M$, unlike the statements above.

\begin{proof}
Suppose $M$ is a subgroup of $\GLn$ and $H = M/\Gamma$. Then $H$ acts on $\Gamma$ by conjugations by elements of preimage. It is an action as $\Gamma$ is abelian. Consequently there exists a homomorphism $\varphi \colon H \to \mathrm{Aut}(\Gamma) \cong \mathbb{Z}/2\mathbb{Z}$, and so there is a subgroup of $H$ of index less or equal to two, that acts on $\Gamma$ identically. Note that if we prove Theorem \ref{group_main} for this subgroup, it would be proved for the whole group as the index increases no more than twice. Note that $H$ acts on $\Gamma$ identically if and only if $\Gamma$ is a subgroup of $Z(M)$. So now we have to prove the theorem only in case where $\Gamma$ lies in $Z(M)$.

By Proposition \ref{group_add2} the index $c_M$ is equal to the smallest index of an abelian subgroup of some finite subgroup of $\GLn$. The rest follows from the fact that $\GLn$ is Jordan (see Theorem~\ref{GLn}).
\end{proof}

The next proposition will help us to describe the structure of Hopf manifolds.

\begin{predl} \label{for}
Suppose $M$ is a group, and $\Gamma$ is a normal subgroup of $M$ of finite index isomorphic to $\mathbb{Z}$. Then there exists a characteristic subgroup $\Theta$ of $M$ such that $\Theta \cong \mathbb{Z}$ and $\Theta \subset \Gamma$.
\end{predl}

\begin{proof}
Suppose $\mu$ is the canonical projection from $M$ to $M / \Gamma$, and $n$ is an index of $\Gamma$ in $M$. Let $k$ be a generating element of $\Gamma$. Let's prove that the group generated by $k^n$ is a characteristic subgroup of $M$. Let $\sigma$ be an automorphism of $M$. Note that $|\im \mu| = n$, so $$\mu(\sigma(k^n)) = \mu(\sigma(k)^n) = \mu(\sigma(k))^n = e,$$ and so  $\sigma(k^n) \in \Gamma$. Suppose $\sigma(k^n) = k^t$. Then the index of a subgroup generated by $\sigma(k^n)$ in $M$ is equal to $nt$. But the index of a group generated by $k^n$ in $M$ is equal to $n^2$ and so $t$ is equal to $n$ or $-n$. So we proved that the group generated by $k^n$ maps to itself by any automophism of $M$. So it is characteristic.
\end{proof}

\section{Hopf manifolds}

In this section we discuss the structure of Hopf manifolds. Here we don't prove anything new, all the results are taken from the articles of Kato \cite{hoph} and Kodaira \cite{kod}. However some of the proofs in those articles are not explicit enough, so we have chosen to reformulate them here.

\begin{definition} [{see \cite[Introduction]{hoph}}] \label{contraction}
An automorphism $\varphi \colon \Co^n \to \Co^n$ is called a contraction, if it preserves $0$ and satisfies the following two conditions: $\lim_{\nu\rightarrow+\infty}\varphi^{\nu}(x)=0$ for every $x\in \Co^n$; and for any small neighborhood $U$ of zero there exists $\nu_0 \in \mathbb{N}$ such that $\varphi^\nu (U) \subset U$ for all $\nu \geqslant \nu_0$.
\end{definition}

The differential of a contraction $d\varphi(0)$ is an invertible matrix with absolute values of all the eigenvalues less than $1$ (see \cite[Proof of Lemma 13]{Kato2}).

\begin{definition} [{see \cite[Introduction]{hoph}}] \label{primary}
A compact complex manifold is called a primary Hopf manifold if it is isomorphic to $(\Co^n \backslash 0) / \Gamma$, where $\Gamma \cong \mathbb{Z}$ is a group generated by a contraction.
\end{definition}

We remind the reader the definition of properly discontinuous action:

\begin{definition}[{see \cite[Definition 8, Theorem 9 (2)]{propdisc}}] \label{prop1}
The action of a group $M$ on a topological space $X$ is called properly discontinuous if for any compact $K \subset X$ the set $\{ g \in M \, | \, g(K) \cap K \ne \emptyset\}$ is finite.
\end{definition}

Suppose $X$ is a complex manifold, and $M$ acts on $X$ so that the quotient of $X$ by the action of $M$ is a compact complex manifold. Then the action of $M$ is properly discontinuous.

The proof of the next lemma is taken from \cite[pages 694-695]{kod} and \cite[pages 48-49]{hoph}.

\begin{statement} \label{el}
Suppose $Y$ is a Hopf manifold and $M$ is a subgroup of the group of automorphisms of $\Co^n \backslash 0$ such that $Y = (\Co^n \backslash 0) / M$. Then $M$ contains a contraction $g$, such that the group generated by $g$ is a normal subgroup of $M$ of finite index.
\end{statement}

\begin{proof}
By Hartogs theorem $M$ acts on $\Co^n$ preserving $0$. Suppose $B$ is a closed unit ball in $\Co^n$ and $S$ is its boundary. Consider $g \in M$ such that $g(S) \cap S = \emptyset$. It exists as $M$ acts properly discontinuously on $\Co^n \backslash 0$. Then either $B \backslash S$ contains $g(B)$, or $g(B \backslash S)$ contains $B$. We can consider only the first variant as we can get it from the second one by replacing $g$ by $g^{-1}$. We get that $$g^k(B) \subset g^{k-1}(B) \backslash N_{k-1},$$ where $N_{k-1}$ is a boundary of $g^{k-1}(B)$. Let's prove that $$\underset{k \geqslant 1}{\bigcap} \, g^k (B) = 0.$$

We will call a norm of a compact set the maximal distance from a point of the set to zero. Let $m_i$ be the norm of $g^i(B)$. Then, as $g^k(B) \subset g^{k-1}(B)$, the norms $m_i$ form a decreasing sequence. Let $m$ be its' limit. Suppose $m \ne 0$. Let $D$ be the set of all points $d$ such that $m \leqslant ||d|| \leqslant 1$. Then $D$ is compact and $D \cap g^i(D) \ne 0$ for every $i$, as the norm of $N_i \subset g^i(D)$ is equal to the norm of $g^i(B)$ and so is greater than $m$. But it is a contradiction with Definition \ref{prop1}. So $m = 0$ and then $$\underset{k \geqslant 1}{\bigcap} \, g^k (B) = 0.$$ 

%Let $z$ be a nonzero point on the boundary of $\underset{k \geqslant 1}{\bigcap} \, g^k (B)$. Let $t$ be an arbitrary point in $S$. By Definition \ref{prop2} there exist neighborhoods $U_t$ and $V_t$ of $t$ and $z$ correspondingly, such that the set $\{ g \in M \, | \, g(U_t) \cap V_t \ne \emptyset\}$ is finite. The set $S$ is compact so from the cover of $S$ by neighborhoods $U_t$ we can choose a finite subcover. Denote it by $U_{t_1}, \ldots U_{t_l}$. Let $$V = \underset{i}{\bigcap} \, V_i,$$ where $i \in \{t_1,\ldots,t_l\}$. Note that $V$ is a neighborhood of $z$ and $z$ lies on the boundary of the intersection of the images of $B$, and so $V$ is not a subset of $g^k(B)$ for all large enough $k$. However $z$ lies in $g^k(B)$, and so $V$ intersects with $g^k(S)$ for all large enough $k$. But it is a contradiction with the definition of $V$. So $$\underset{k \geqslant 1}{\bigcap} \, g^k (B) = 0.$$ 

Now we will prove that $g$ is a contraction. First we prove that for every small neighborhood of zero $U$ there exists $\nu_0 \in \mathbb{N}$ such that $\varphi^\nu (U) \subset U$ for all $\nu \geqslant \nu_0$. Suppose $U$ is a subset of $B$. Then as $$\underset{k \geqslant 1}{\bigcap} \, g^k (B) = 0,$$ there exists $\nu_0$ such that $g^{\nu_0}(B) \subset U$, and so $g^\nu(U) \subset U$ for all $\nu \geqslant \nu_0$. It is left to prove that $\lim_{\nu\rightarrow+\infty}\varphi^{\nu}(x)=0$ for every $x \in \Co^n$. Note that for points inside $B$ we've already proved that. Now suppose $x$ is an arbitrary point, $S'$ is a sphere in $\Co^n$ with the center $0$ that contains $x$. Let $D$ be the compact part of $\Co^n$ which lies between $S$ and $S'$. By Definition \ref{prop1} there exists $r$ such that $g^r(D) \cap D = \emptyset$. But also, as we proved above, $g^r(B) \subset B$ for all $r$. So $g^r(x) \in B$, and for every point $y$ of $B$ we've proved that $\lim_{\nu\rightarrow+\infty}g^\nu y=0$. So $\lim_{\nu\rightarrow+\infty}g^{\nu}x=0$.

Let $\Gamma$ be a subgroup generated by $g$. We will prove that $\Gamma$ contains a normal subgroup of $M$ of finite index. As $\Gamma$ is a subgroup of $M$, there exists a commutative diagram

$$\xymatrix@1{
\Co^n \backslash 0 \ar[r] \ar[rd] & (\Co^n \backslash 0)/\Gamma \ar[d]\\
& Y
}$$

where both $(\Co^n \backslash 0)/\Gamma$ and $Y$ are compact. Note that there is no infinite covering maps between compact complex manifolds, so the index of $\Gamma$ in $M$ is finite. So there exists a normal subgroup of $M$ that is contained in $\Gamma$. This is a normal subgroup of finite index generated by a contraction.
\end{proof}

So we proved the following statement:

\begin{corollary} Every Hopf manifold can be obtained from a primary Hopf manifold by taking a quotient by a finite group.
\end{corollary}

Note that the quotient map from the corollary above is not unique. However there are the preferable ones. We will prove that for every Hopf manifold $Y$ there exists a primary Hopf manifold $X$ such that all the automorphisms of $Y$ can be raised to automorphisms of $X$.

\begin{lemma} \label{aut}
For every Hopf manifold $Y$ there exists a primary Hopf manifold $X$ such that $Y = X / G$, where $G$ is a finite group, and there exists an exact sequence of groups $$1 \to G \to \widehat{\mathrm{Aut}}(Y) \to \mathrm{Aut}(Y) \to 1,$$ where $\widehat{\mathrm{Aut}}(Y) \subset \mathrm{Aut}(X)$.
\end{lemma}

\begin{proof}

Let $Y$ be a Hopf manifold. Suppose $M$ is a subgroup of the automorphism group of $\Co^n \backslash 0$ such that $Y = (\Co^n \backslash 0) / M$. Let $g \in M$ be such a contraction that the group $\Gamma$ generated by $g$ is normal, and has finite index (see Lemma \ref{el}). Then by Proposition \ref{for} in $M$ exists a characteristic subgroup $\Theta \subset \Gamma$ such that $\Theta \cong \mathbb{Z}$. 

Let $X = (\Co^n \backslash 0) / \Theta$. Then $X$ is a primary Hopf manifold. Denote by $G$ the group $M / \Theta$. Then $G$ is finite and $Y = X/G$. 

As $\Co^n \backslash 0$ is the universal cover of $Y$, there exists an exact sequence of groups
$$1 \to M \to \widetilde{\mathrm{Aut}}(Y) \to \mathrm{Aut}(Y) \to 1,$$
where $\widetilde{\mathrm{Aut}}(Y)$ is the subgroup of $\mathrm{Aut}(\Co^n \backslash 0)$. As $\Theta$ is a characteristic subgroup of $M$, the group $\Theta$ is normal in $\widetilde{\mathrm{Aut}}(Y)$. So there exists an exact sequence 
$$1 \to \Theta \to \widetilde{\mathrm{Aut}}(Y) \to \widehat{\mathrm{Aut}}(Y) \to 1,$$
where $\widehat{\mathrm{Aut}}(Y) \subset \mathrm{Aut}(X)$. From the last two sequences we get the exact sequence
$$1 \to G \to \widehat{\mathrm{Aut}}(Y) \to \mathrm{Aut}(Y) \to 1.$$ \\[-35pt]
\end{proof}

\section{Proof of the main theorem}

For a point $P$ of a manifold $X$ we denote by $T_{P,X}$ the tangent space to $X$ in $P$. 

We remind the reader a general fact about finite subgroups of the group of automorphisms of a manifold.

\begin{theorem} [{see \cite[\textsection 2.2]{theorem}}] \label{th}
Suppose $X$ is an irreducible complex manifold and $H \subset \mathrm{Aut}(X)$ is a finite group, that preserves $P \in X$. Then the natural homomorphism $H \to \mathrm{GL}(T_{P,X})$ is an injection.
\end{theorem}

Now we can prove Theorem \ref{main}.

We will prove it for primary Hopf manifolds first.

\begin{lemma} \label{s_main}
The automorphism group of a primary Hopf manifold is Jordan.
\end{lemma}

\begin{proof}
Let $X$ be a primary Hopf manifold. Note that $X$ is a quotient of $\Co^n \backslash 0$ by an action of a group $\Gamma$ which is isomorphic to $\mathbb{Z}$. As $\Co^n \backslash 0$ is a universal cover of $X$ we conclude that there exists an exact sequence of groups $$1 \to \Gamma \to \widetilde{\mathrm{Aut}}(X) \to \mathrm{Aut}(X) \to 1,$$ where $\widetilde{\mathrm{Aut}}(X)$ is a subgroup of $\mathrm{Aut}(\Co^n \backslash 0)$. By Hartogs theorem the automorphisms of $\Co^n \backslash 0$ can be extended to the automorphisms of $\Co^n$, so we can say that $\widetilde{\mathrm{Aut}}(X)$ acts on $\Co^n$, preserving $0$.

Suppose $H$ is a finite subgroup of $\mathrm{Aut}(X)$, and $M$ is its preimage in $\widetilde{\mathrm{Aut}}(X)$. Then $\Gamma$ injects to $M$ as a kernel of the corresponding map. Consider the natural homomorphism $$\varsigma \colon M \to \mathrm{GL}(T_{0,\Co^n}) \cong \GLn.$$ The deferential of the map which generates $\Gamma$ is invertilble and has nonzero eigenvalues of absolute value less than $1$, so $\varsigma|_\Gamma$ is an injection. The group $\Gamma$ has finite index in $M$, so the kernel of $\varsigma$ is finite, and so, by Theorem \ref{th}, it is trivial. Consequently $\varsigma$ is an injection on $M$. By Theorem \ref{group_main} there exists a constant $C$ which depends on $n$ (but not on $\Gamma$ or $M$), such that the smallest index of abelian subgroups of $H = M/\Gamma$ is less than $C$. Consequently $\mathrm{Aut}(X)$ is Jordan.
\end{proof}

Now let's prove Theorem \ref{main} for any Hopf manifold.

\begin{proof}
By Lemma \ref{aut} for any Hopf manifold $Y$ there exists a primary Hopf manifold $X$ such that there exists an exact sequence $$1 \to G \to \widehat{\mathrm{Aut}}(Y) \to \mathrm{Aut}(Y) \to 1,$$ where $\widehat{\mathrm{Aut}}(Y)$ is a subgroup of $\mathrm{Aut}(X)$ and $G$ is a finite group.

By Lemma \ref{s_main} the group $\mathrm{Aut}(X)$ is Jordan. So its subgroup $\widehat{\mathrm{Aut}}(Y)$ is also Jordan, and so its quotient by a finite subgroup is Jordan. So we proved that $\mathrm{Aut}(Y)$ is Jordan.
\end{proof}

{\scshape National Research University
\enquote{Higher School of Economics}. 
Faculty of mathematics.} \\[-5pt]

{E-mail: AnnSavelyeva57@gmail.com}
\end{document}